\documentclass[11pt,oneside,leqno]{amsart}

\usepackage{amsxtra}
\usepackage{amsopn}
\usepackage{amsmath,amsthm,amssymb, enumerate}
\usepackage{amscd}
\usepackage{pifont}
\usepackage{amsfonts}
\usepackage{latexsym}
\usepackage{verbatim}
\usepackage{pb-diagram}

 \newcommand{\nc}{\newcommand}
 
\newcommand\g{{\mathfrak{g}}}
\newcommand\h{{\mathfrak{h}}}

 \renewcommand{\aa}{\mathfrak{a} } \newcommand\aff{{\mathfrak{aff}}}
\nc{\bb}{\mathfrak{b} }
 \nc{\cc}{\mathfrak{c} }  \nc{\dd}{\mathfrak{d} } 
    \nc{\ggo}{\mathfrak{g} }
 \nc{\hh}{\mathfrak{h} }  \nc{\ii}{\mathfrak{i} }
 \nc{\jj}{\mathfrak{j} }  \nc{\kk}{\mathfrak{k} }
\nc{\mm}{\mathfrak{m} }   \nc{\nn}{\mathfrak{n} }
 
\nc{\pp}{\mathfrak{p} }   
\nc{\rr}{\mathfrak{r} } \nc{\sg}{\mathfrak{s} }
 \nc{\sso}{\mathfrak{so} }  \nc{\spg}{\mathfrak{sp} }
 \nc{\sug}{\mathfrak{su} }  \nc{\ssl}{\mathfrak{sl} }
 \nc{\tog}{\mathfrak{t} }  \nc{\uu}{\mathfrak{u} }
 \nc{\vv}{\mathfrak{v} } \nc{\ww}{\mathfrak{w} }
 \nc{\zz}{\mathfrak{z} }

\nc{\CC}{{\mathbb C}}
 \nc{\DD}{{\mathbb D}}
\nc{\FF}{{\mathbb F}}
\nc{\GG}{{\mathbb G}}  
\nc{\HH}{{\mathbb H}}
\nc{\II}{{\mathbb I}}
\nc{\JJ}{{\mathbb J}}
\nc{\KK}{{\mathbb K}}
\nc{\NN}{{\mathbb N}}

\nc{\RR}{{\mathbb R}}  
 \nc{\ZZ}{{\mathbb Z}}

\nc{\ggob}{\overline{\mathfrak{g}}} 
 
\nc{\glg}{\mathfrak{gl} }
  
\nc{\pca}{\mathcal{P}} \nc{\nca}{\mathcal{N}}
 
 \nc{\vp}{\varphi} \nc{\ddt}{\frac{{\rm d}}{{\rm d}t}}
 \nc{\la}{\langle} \nc{\ra}{\rangle}
 
 \nc{\SO}{{\sf SO}} \nc{\Spe}{{\sf Sp}} \nc{\Sl}{{\sf Sl}}
 \nc{\SU}{{\sf SU}} \nc{\Or}{{\sf O}} \nc{\U}{{\sf U}}
 \nc{\Gl}{{\sf Gl}} \nc{\Se}{{\sf S}} \nc{\Cl}{{\sf Cl}}
 \nc{\Spin}{{\sf Spin}} \nc{\Pin}{{\sf Pin}}
 \nc{\Ta}{{\sf T}}
 
 \nc{\Id}{\operatorname{Id}} 
 
 \nc{\ad}{\operatorname{ad}} \nc{\Ad}{\operatorname{Ad}}
 \nc{\coad}{\operatorname{coad}}
 \nc{\rank}{\operatorname{rank}} \nc{\Irr}{\operatorname{Irr}}
\nc{\Ima}{\operatorname{Im}}
 \nc{\End}{\operatorname{End}} \nc{\Aut}{\operatorname{Aut}}
\nc{\Aff}{\operatorname{Aff}}
\nc{\GL}{\operatorname{GL}}
 \nc{\Inn}{\operatorname{Inn}} \nc{\Der}{\operatorname{Der}}
 \nc{\Ker}{\operatorname{Ker}} \nc{\Iso}{\operatorname{I}}
 \nc{\Le}{\operatorname{L}} \nc{\tr}{\operatorname{tr}}
 \nc{\dif}{\operatorname{d}} \nc{\sen}{\operatorname{sen}}
 \nc{\modu}{\operatorname{mod}} \nc{\Ric}{\operatorname{R}}
 \nc{\Sym}{\operatorname{Sym}} \nc{\sca}{\operatorname{sc}}
 \nc{\scalar}{{\sf s}} \nc{\grad}{\operatorname{grad}}
 \nc{\ricci}{\operatorname{r}} \nc{\riccin}{\operatorname{Ric}}
 \nc{\Lie}{\operatorname{L}} \nc{\ct}{\operatorname{T}}

 \theoremstyle{plain}
 \newtheorem{thm}{Theorem}[section]
 \newtheorem{prop}[thm]{Proposition}
 \newtheorem{cor}[thm]{Corollary}
 \newtheorem{lem}[thm]{Lemma}
 
 \theoremstyle{definition}

 \newtheorem{rem}[thm]{Remark}
 \newtheorem{exa}[thm]{Example}

\begin{document}


\title[From almost (para)-complex structures  to affine structures]
{From almost (para)-complex structures to affine structures on Lie groups}

\author[G. Calvaruso]{Giovanni Calvaruso}
\address{Giovanni Calvaruso: Dipartimento di Matematica e Fisica \lq\lq E. De Giorgi\rq\rq \\
Universit\`a del Salento\\
Prov. Lecce-Arnesano \\
73100 Lecce\\ Italy.}
\email{giovanni.calvaruso@unisalento.it}

\author[G. P. Ovando]{Gabriela P. Ovando}
\address{G. P. Ovando: CONICET and ECEN-FCEIA, Universidad Nacional de Rosario \\
Pellegrini 250, 2000 Rosario, Santa Fe, Argentina}
\email{gabriela@fceia.unr.edu.ar}



\begin{abstract}   
 Let $G=H\ltimes K$ denote a semidirect product Lie group with Lie algebra $\ggo=\hh \oplus \kk$, where $\kk$ is an ideal and $\hh$ is a subalgebra of the same dimension as $\kk$. There exist some natural split isomorphisms $S$ with $S^2=\pm \Id$ on $\ggo$: given  any linear isomorphism $j:\hh \to \kk$, we get  the almost complex structure $J(x,v)=(-j^{-1}v, jx)$ and  the almost paracomplex structure $E(x,v)=(j^{-1}v, jx)$. In this  work we show that  the  integrability of the structures $J$ and $E$ above is equivalent to the existence of a left-invariant torsion-free connection $\nabla$ on $G$ such that $\nabla J=0=\nabla E$ and also to the existence of an affine structure on $H$. Applications  include complex, paracomplex  and symplectic geometries. 
\end{abstract}

\thanks{{\em Acknowledgements:} G. Ovando thanks  G. Calvaruso and the Universit\`a del Salento for the invitation and hospitality during her visit. \\
G. Calvaruso was partially supported by funds of the University of Salento and MIUR (within PRIN).
\\
G. Ovando was partially supported by ANPCyT, CONICET and SeCyT - Universidad Nacional de Rosario, Argentina.}


\maketitle

\section{Introduction}
 Given a real vector space $V$, a linear isomorphism $S\in \End(V)$  satisfying $S^2 = \lambda  \Id$, $\lambda = \pm 1$ is called {\em split} if it has exactly
two eigenspaces (of the complexification of the vector space, if $\lambda  = −1$) with the same
dimension. For instance,  on $\RR^{2n}$ both the classical complex structure
$$(x_1, x_2, \hdots, x_n,y_1, y_2, \hdots, y_n) \quad \mapsto \quad (-y_1,-y_2, \hdots, -y_n, x_1, x_2, \hdots, x_n),$$
and the classical involution
$$(x_1, x_2, \hdots, x_n,y_1, y_2, \hdots, y_n) \quad \mapsto \quad (y_1,y_2, \hdots, y_n, x_1, x_2, \hdots, x_n);$$
are split isomorphisms. 

These structures give rise to almost complex and almost paracomplex geometries on manifolds, which are induced by differential (1,1)-tensors on a differentiable manifold $M$. Both complex and paracomplex geometries are very active research fields (see in particular the survey \cite{CFG} and references therein for the last one). Other topics  related to them are, among others,  (almost) K\"ahler, para-K\"ahler,  bi-lagrangian, product and complex product structures and more generally, generalized complex and paracomplex geometry. Up to our knowledge, the first one to consider generalized complex and paracomplex structures simultaneously in a systematic way was  Vaisman in \cite{Va}. The aim of this paper is to  consider simultaneously a kind of complex and paracomplex structures on Lie groups.

Let $G=H \ltimes K$ be a semidirect product Lie group and denote by $\ggo$, $\hh$ and $\kk$ the corresponding Lie algebras.  Then, $\hh$ and $\kk$ are respectively a Lie subalgebra and an ideal of $\ggo$, inducing the natural splitting  $\ggo=\hh \oplus \kk$, which corresponds to the representation $\pi:\hh\to \End(\kk)$ acting by derivations, and is referred to in literature as the {\em semidirect product} (or {\em sum}) of $\hh$ and $\kk$. This algebraic setting was considered for almost contact structures in \cite{Ou}.

Suppose now that $\dim H = \dim K$  and consider a linear isomorphism $j:\hh\mapsto \kk$. Then, we can define on $\ggo$ a corresponding 
 almost complex structure $J$ and almost paracomplex structure $E$, respectively determined by 
$$J(x,v)=(-j^{-1} v, jx) \qquad \text{and} \qquad E(x, v)=(j^{-1} x, j v), \quad \mbox{ for all } x\in \hh, v\in \kk.  
$$
Thus, $\hh$ and $\kk$ are complementary and totally real subspaces for $J$.  

Moreover, because of its algebraic decomposition,  the semidirect product Lie group $G=H \ltimes K$  can be described in terms of a canonical left-invariant paracomplex structure $F$, 
 which at the level of the Lie algebra $\ggo$ is  defined by
$$F(x,v)=(x, -v),\quad \mbox{ for all } x\in \hh, v\in \kk.$$

The following  natural questions motivate our research in this context:

\begin{enumerate}[(i)] {\em 
\item  How to characterize the integrability of the split structures  $J$ and   $E$?
\item Is there any distinguished connection  for $J$ or $E$ on $\ggo$?  More explicitly, when there exists a torsion-free connection on $\g$, such that $\nabla J=0=\nabla E$?}
\end{enumerate}

The results here  can also be  seen as a kind of generalization of the following known fact  (see \cite{Chu}). A symplectic structure $\omega$ on a Lie algebra $\ggo$ gives rise to an affine structure  $\nabla$ on $\ggo$,  implicitly defined by
$$\omega(\nabla_x y, z)=\omega(y,[x,z]),$$
so that we get a converse in a sense we shall explain. 
 In the present work we study the above questions taking into account the algebraic setting underlying the semidirect product structure on $G$. Notice that the condition $F=E J$ is always satisfied. Consequently, both pairs $\{J,E\}$  and $\{J,F\}$ give rise to a left-invariant {\em almost complex product structure} on $G$ (see \cite{AS}).

The results we obtain along the article  prove the following equivalences: 

\begin{enumerate}[(i)]
\item {\em The almost complex structure $J$ is integrable;} 
\item {\em The almost complex structure $E$ is integrable;} { (Theorem 3.2)}.
\item {\em There exists a left-invariant torsion-free connection $\nabla$ on $G$ such that $\nabla J=0$ and $\nabla E=0$;}
 {(Theorem 4.5')}.
\item {\em There exists a compatible left-invariant affine connection on $H$;} { (Theorem 5.8}).
\end{enumerate}

The key  to prove this is based on the structure of $\kk$ and the linear map $j$. We prove  that {\em all the statements above are  equivalent to the fact that $\kk$ is abelian and $j:\hh\to \kk$ is a 1-cocycle of $(\hh, \pi)$}. In particular, once we know the representation, the integrability condition for $J$ and $E$ on $\ggo=\hh\oplus \kk$ becomes a linear equation. 

Affine structures became relevant since the  fundamental question of Milnor concerning the existence of complete affine structures on solvable Lie groups (see \cite{Mi} for details).  Making use of the above equivalences, one can  get  examples of totally real complex structures, paracomplex structures and  left-symmetric affine (LSA) structures in a direct way. We remark that in this framework the integrability condition for $J$  follows from a linear equation, once  the representation is given. Several further implications follow once we look more deeply  at the representation. Canonical examples where the situation above appears are both the tangent and the cotangent bundle of a given Lie group. Consequently, almost complex,  generalised complex and almost paracomplex geometries are all touch\'e by the results above. 
 
Finally, two  remarks  concerning these results:

\begin{itemize}
\item They generalize  previous results concerning complex product structures \cite{AS, BV},  complex and symplectic structures related to tangent algebras \cite{ABD,BD,COP}, complex and paracomplex structures on homogeneous manifolds \cite{CF}.

\item The existence of LSA structures imposes a clear obstruction. In fact,   the Lie algebras $\hh$ and  $\ggo$ { are necessarily} solvable, see \cite{Au,Mi}.
\end{itemize}

 The paper is organized in the following way. We introduce totally real almost complex and paracomplex structures on a {semidirect product Lie algebra in Section~2, and in Section~3 we characterize their integrability. In Section~4 we investigate the existence of a torsion-free connection for which the almost complex structure $J$ and almost paracomplex structure $E$ are  parallel. In Section~5 we  consider LSA structures and prove the last equivalence. Several examples will be described in Section~6, showing explicitly the construction of compatible LSA from totally real complex structures on a semidirect product Lie algebra and, conversely, given a Lie algebra $\hh$ with a compatible LSA, the construction of a totally real complex structure on semidirect product Lie algebra $\g =\hh \oplus \kk$.  
Finally, we show how an arbitrary inner product on $\h$ extends on $\g =\hh \oplus \kk$ to an inner product compatible with the $J$ above and to neutral metrics compatible with the $E$ and $F$ above, and give a condition ensuring the integrability of these almost 
K\"ahler and para-K\"ahler structures.

\section{Natural split isomorphisms for semidirect products}

 An {\em almost  product structure} on a Lie  group $G$ is determined by a  $(1,1)$-tensor $P$ on  $G$ such that $P^2=\Id$, but $\tilde P\neq \pm \Id$, which  gives rise to a splitting of the tangent bundle $T G$ into the
Whitney sum of two subbundles $T^{\pm} G$ corresponding to the {($\pm 1$)-eigenspaces} of $P$. 
 For a Lie group $G$, one usually asks the  almost product structure to be invariant under translations on the left by elements $g\in G$, so that  the $(1,1)$-tensor is determined at the Lie algebra $\ggo$ by   the endomorphism $P:\ggo \to \ggo$ satisfying $P^2=\Id$ but $P\neq \pm \Id$. 
{The {\em integrability} of such $P$ is expressed by the vanishing of its Nijenhuis tensor: $N_{P}\equiv 0$,} where
\begin{equation}\label{NijenhuisE}
N_{P}(Y,Z)=[P Y, P Z]+[Y,Z]-P[P Y,Z]-P[Y,P Z]\qquad \text{for all} \;\, Y,Z  \in \g .
\end{equation}
This integrability condition is equivalent to the fact that the eigenspaces of $P$ are subalgebras of $\ggo$. Globally the distributions determined by these eigenspaces by translations on the left, are completely integrable. Almost product structures which are integrable are simply called {\em product structures}.

If $\dim \hh=\dim \kk$, the almost product structure on $\ggo$ is split and it is called an {\em almost paracomplex structure}. When integrable, it is just called a {\em paracomplex structure}.

Product structures arise naturally in an algebraic context by
 considering the semidirect product of two  groups $H$ and $K$, via a group homomorphism $\pi: H \to \Aut(K)$. This  gives a decomposition 
$\ggo=\hh \oplus \kk$ into a direct sum as vector spaces, where $\hh$ is a subalgebra of $\g$, $\kk$ is an ideal of $\ggo$ and $\pi:\hh\to \End(\kk)$ is a representation acting by derivations.

Explicitly, the  Lie bracket on $\ggo$ satisfies 
\begin{equation}\label{corchete}
\begin{array}{rcll}
[x,y] & = & [x, y]_{\hh} & \text{for all} \;\, x, y\in \hh, \\[4pt]
{[x,v]} & = & \pi(x) v & \text{for all} \;\, x\in \hh, v\in \kk, \\[4pt]
 [u,v] & = & [u,v]_{\kk} & \text{for all} \;\, u,v \in \kk,
\end{array}
\end{equation}
where $[\cdot, \cdot]_{\hh}=\hh \times \hh \to \hh$ denotes the Lie bracket on $\hh$ and correspondingly for $[\cdot, \cdot]_{\kk}$. The Jacobi identity on $\ggo$ implies that $\pi:\hh \to \End(\kk)$ is a representation acting by derivations. 
Conversely, starting with a pair of Lie algebras $(\hh, [\cdot,\cdot]_{\hh})$ and $(\kk,[\cdot, \cdot]_{\kk})$  and a representation $\pi:\hh\to \End(\kk)$  acting by derivations,  the direct sum as vector spaces $\ggo=\hh\oplus \kk$  equipped with the binary operation in (\ref{corchete}) introduces  a Lie algebra structure on $\ggo$.  Abusing of names we  refer to such  Lie algebra $\ggo$  as the  {\em semidirect product Lie algebra} of $\hh$ and $\kk$ via $\pi$.  In this situation, the following exact sequence of Lie algebras  splits:
$$0 \quad  \longrightarrow \quad \kk  \quad \longrightarrow \quad \ggo \quad \longrightarrow \quad  \hh \quad \longrightarrow 0.$$ 

\noindent
For $\kk$  abelian and $\dim \kk=\dim \hh$,  the resulting semidirect product Lie algebra $\ggo=\kk \oplus \hh$ via the representation $\pi$ is sometimes called the {\em tangent Lie algebra} $\ct_{\pi} \hh$.

\begin{exa} \label{tan-co} Both the tangent and cotangent  bundles of a given Lie group $H$ admit a Lie group structure {which can be described in terms of semidirect products}.  In fact:
\begin{itemize}

\item The tangent bundle  $TH$  of a Lie group $H$ is identified with $H \times \hh$ as a semidirect product Lie group with Lie algebra $\ct \hh=\hh \oplus \hh$, the second copy of $\hh$ with the abelian Lie bracket and  the adjoint representation $\ad:\hh\to \End(\hh)$ as $\ad(x) v=[x,v]$ for $x,v\in\hh$. 

\item For the cotangent bundle of the Lie group $H$, one identifies $T^*H$ with $H\times \hh^*$, where one has the coadjoint representation $\ad^*:\hh \to \End(\hh^*)$ given  $\ad(x)^*\varphi (y)= -\varphi \circ \ad(x)(y)$ for all $x,y\in\hh$, $\varphi\in\hh^*$. It is usually denoted by $\ct^*\hh$   and called the cotangent Lie algebra of $\hh$. 

\end{itemize}
\end{exa}

A real linear isomorphism $S$ on a Lie algebra $\ggo$ satisfying $S^2 = - Id$ is an  {\em almost complex structure}. 

   Let  $J$ be almost complex. It  is said {\em  integrable} or simply a  {\em complex structure}  if and only if $N_J\equiv 0$ where 
\begin{equation}\label{NijenhuisJ}
N_J(Y,Z)=[JY,JZ]-[Y,Z]-J[JY,Z]-J[Y,JZ]\qquad \text{for all} \;\, Y,Z \in \g. 
\end{equation}
In this context, the integrabiliy condition of $J$ is equivalent to the fact that both eigensubspaces for $J$ (as a complex linear map) are Lie subalgebras of the complexification $\ggo^{\CC}$.

\begin{exa} Up to isomorphisms, there exist two non-isomorphic real Lie algebras in dimension two. One is the real abelian Lie algebra underlying $\CC$ which is a complex manifold and the other one,  denoted by $\aff(\RR)$ as it is Lie algebra underlying the group of affine motions of $\RR$, is spanned by the vectors $x,y$ satisfying the Lie bracket $[x,y]=y$. In the
abelian situation we also have a basis $x,y$ with $[x,y]=0$. 
In both cases, there exists a complex structure  defined by
$$Jx=y \qquad Jy=-x.$$ 
Observe that both  $\vv_1=\RR x$ and $\vv_2=\RR y$ are totally real subspaces of the semidirect product $\vv_1\oplus \vv_2$.

\end{exa}

A semidirect product Lie algebra $\ggo=\hh\oplus \kk$ has a natural product structure associated to this decomposition, given by  the  linear map 

\begin{equation}\label{defF}
F(x,v)=(x, -v), \quad x\in \hh, v\in \kk.
\end{equation}
%

Assume that $\dim \hh=\dim \kk$, so that the map $F$ above constitutes a paracomplex structure on $\ggo$. 

{A  linear isomorphism $j: \hh \to \kk$ gives rise to split isomorphisms. In fact, it  defines }
\begin{itemize}
\item an almost complex structure $J$ on $\ggo$,  by 
\begin{equation}\label{defJ}
J(x, v) = (-j^{-1}v, jx)\quad \mbox{ for all } x\in \hh, v\in \kk.
\end{equation}
\item an almost paracomplex structure $E:\ggo  \to \ggo$, by
\begin{equation}\label{defE}
E(x,v)=(j^{-1}v, jx), \quad x\in \hh, v\in \kk.
\end{equation}
\end{itemize}

We  observe that for such $J$, the subspaces $\hh$ and $\kk$ are totally real, that is, they satisfy $J\hh\cap \hh=\{0\}$ and the same holds for $\kk$. Such an almost complex structure was called  {\em totally real} with respect to the decomposition $\ggo=\hh\oplus \kk$ in \cite{COP}. 

Conversely, if an almost complex structure on $\ggo=\hh\oplus \kk$ satisfies $J\h =\kk$, then the map $j:=J|_{\h}: \h \to \kk$ is a linear isomorphism.

It is easy to check that  the  ($\pm 1$)-eigenspaces of the almost paracomplex structure $E$ are the subspaces given by
$$E_{\pm 1}=\{(x,\pm jx) \quad : \quad x\in \hh\}.$$
Again, $E\h=\kk$ and we   call $E$  a {\em totally real almost paracomplex structure}.

\begin{exa} We note that a \lq\lq totally real complex structure\rq\rq \ $J$ always underlies a fixed decomposition on a given Lie algebra. 
For instance, take the real Lie algebra $\ct^* \aff(\RR)$   with decomposition $\aff(\RR)\oplus \aff(\RR)^*$, where $\aff(\RR)=span\{e_1, e_2\}$ with $[e_1,e_2]=e_2$ and $\aff(\RR)^*=span\{e_3,e_4\}$. From the coadjoint action one gets $e_4=-[e_1,e_4]=[e_2,e_4]$. On $\ct^* \aff(\RR)$ consider the complex structure given by 
$$Je_1=e_2\qquad Je_3=-e_4. $$ 
This $J$ is not totally real with respect to the decomposition $\aff(\RR)\oplus \aff(\RR)^*$. {However,} it is totally real for  the decomposition $\hh\oplus \kk$, where  $\hh=span\{e_1, e_3\}$ and $\kk=span\{e_2, e_4\}$.
\end{exa}

 Thus, given the  Lie algebra $\ggo$ with a fixed splitting $\ggo=\hh \oplus \kk$, there is a one-to-one correspondence between:
\begin{itemize}
\item[(i)] totally real almost complex structures on $\ggo$;
\item[(ii)] totally real almost paracomplex structures on $\ggo$;
\item[(iii)] linear isomorphisms $j: \h \to \kk$. 
\end{itemize}

The above characterization shows that the class of totally real almost complex structures (equivalently, of totally real almost paracomplex structures)   on $\ggo=\hh \oplus \kk$, is large and increases quadratically in function of the dimension, since linear isomorphisms $j: \h \to V$ form an open subset of the  vector space of linear maps from $\hh$ to $\kk$ which has  dimension $n^2$, where $n=\dim \h$.

\begin{rem}
 Given an almost complex structure $J$ on a Lie algebra $\ggo$, it is always possible to have a decomposition of $\ggo$ as a direct sum as vector spaces of the form $\ggo=\uu \oplus J\uu$ for which $\uu$ and $J\uu$ are  totally real vector subspaces. 
 For considerations  whenever one only assumes that  $\uu$ is an ideal, we may refer to \cite{CCO}.

On the other hand, given a paracomplex structure on with splitting $\ggo=\uu_1\oplus \uu_2$, where $\uu_i$ are Lie subalgebras of the same dimension, it is always possible to define an almost complex structure $J$ as in Equation (\ref{defJ}) and another almost paracomplex structure as in Equation (\ref{defE}). 
\end{rem}

 Another geometry arising from split isomorphisms is given by  almost complex product structures. Recall that an {\em almost complex product structure} is a pair $(\mathbb J, \mathbb I)$ consisting of 
\begin{itemize}
\item an almost complex structure $\mathbb J$ and
\item an almost paracomplex structure $\mathbb I$,
\end{itemize} such that $\mathbb J \mathbb I = - \mathbb I \mathbb J$. 

Indeed, if the pair $(\mathbb J, \mathbb I)$ is an almost complex product structure, then the pair $(\mathbb J, \mathbb I \mathbb J)$ gives another almost complex product structure. 
We may refer to \cite{AS} for more information on complex product structures.

The pair consisting of the almost complex structure $J$ in Equation  \eqref{defJ} and the almost paracomplex structure $E$ in Equation \eqref{defE} satisfies $JE=-EJ$, and 
$$F= E J$$
gives  the canonical paracomplex structure described in Equation  \eqref{defF}. Therefore, both the pairs  $(J,E)$ and  $(J,F)$ are {\em almost complex product structures} on $\ggo$.

\section{Integrability conditions}

In this section we investigate the integrability conditions for the split isomorphisms $J$ and $E$ defined in Section~2. 
We first observe that for the almost complex and paracomplex structures introduced in Equations \eqref{defJ} and \eqref{defE} and for all $X,Y, Z \in \ggo$, the corresponding Nijenhuis tensors  verify the following relations: 

\medspace

\begin{tabular}{lll}
 $\bullet \qquad N_E(Z,X)=-N_E(X,Z)$ & and & $N_J(Z,X)=-N_J(X,Z)$, \\  
 $\bullet \qquad N_E(EY,EZ)  =  N_E(Y,Z)$ & and & $N_J(JY,JZ)= -N_J(Y,Z)$, \\
  $\bullet \qquad N_E(X,EZ)=-E N_E(X,Z)$ & and & $N_J(Y,JZ)=  -J N_J(Y,Z)$.
	\end{tabular}

\medspace

We emphasize  that although $J$ and $E$ need not be integrable,  the almost paracomplex structure $F$ is always integrable. 

Recall that given two Lie algebras $\hh$, $\kk$ and a representation $\pi:\hh \to \End(\kk)$ acting by derivations,  
a linear map $\theta: \hh \to \kk$ is called a {\em $1$-cocycle} of $(\hh, \pi)$  if
\begin{equation}\label{$1$-cocycle}
\pi(x) \theta(y)-\pi(y) \theta(x) - \theta[x,y]_{\hh}=0.
\end{equation}
 We  refer for instance to \cite{Va} for more details about  the cohomology theory of Lie algebras. 

\begin{rem}  Observe that $1$-cocycles always exist. In fact, given $(\hh,\pi$), it suffices to take a fixed vector $u\in\kk$ and define a linear map $\theta:\hh \to \kk$   by $\theta_u(x)=\pi(x) u$ for all $x\in\hh$. Such a $1$-cocycle is called a coborder. A coborder does not necessarily define an almost complex structure as in Equation (\ref{defJ}) on $\ggo=\hh\oplus \kk$. 
\end{rem}

Cocycles arise naturally  whenever we ask for integrability of the split isomorphisms $J$ and $E$ defined above. In fact, let $j:\hh \to \kk$ denote a linear isomorphism. The Nijenhuis tensor for $J$ 
yields that for all $x,y\in \hh$,   one has
$$
N_J(x, y)  =  [jx, jy]-[x,y]-J[jx,y]-J[x,j y] 
=[jx, jy]-[x,y]-j^{-1}\pi(y) jx +j^{-1}\pi(x) j y.
$$ 
It is clear that due to the relations satisfied by  the Nijenhuis tensor, the condition $N_J\equiv 0$ is equivalent to {requiring that} $N_J(x,y)=0$ for all $x,y\in \kk$. So, this holds if and only if for all $x,y\in \hh$
\begin{itemize}\label{integrability}
\item $[jx, jy]=0$ and
\vspace{4pt}\item $[x,y]+j^{-1}\pi(y) jx -j^{-1}\pi(x) j y=0$,
\end{itemize}
which is equivalent to
\begin{itemize}
\item $\kk$ is abelian and 
\item $j$ is a $1$-cocycle of $(\hh,\pi)$. 
\end{itemize}

More generally, we have the next result.

\begin{thm}\label{intJE} Let $\ggo=\hh\oplus \kk$ denote a semidirect product  Lie algebra. { Given a linear isomorphism $j: \hh \to \kk$, let  $J$ and $E$}  denote the almost complex and paracomplex structures respectively defined in  Equations \eqref{defJ} and \eqref{defE}. The following statements are equivalent:
\begin{itemize}
\item[(i)] the almost complex structure $J$ is integrable; 
\vspace{2pt}\item[(ii)] the almost paracomplex structure $E$ is integrable; 
\vspace{2pt}\item[(iii)] $\kk$ is abelian and the  linear isomorphism $j:\hh \to \kk$ is a $1$-cocycle of $(\hh, \pi)$. 
\end{itemize}
Consequently, when $\kk$ is abelian, the pair $(J,E)$ defines a complex product structure on $\ggo=\hh \oplus \kk$ for any $1$-cocycle $j$ of 
$(\hh, \pi)$.
\end{thm}
\begin{proof}
The equivalence of (i) and  (iii) was proved above. 
By an analogous argument for $E$, given $x, y\in \hh$,  we have
$$
N_E(x, y) =  [jx, jy]+[x,y]+ j^{-1}\pi(y) jx - j^{-1}\pi(x) jy
$$ 
so that $N_E\equiv 0$ is equivalent to $\kk$ abelian and  $j$ is a $1$-cocycle of $(\hh,\pi)$, which proves the equivalence of (ii) and  (iii).
\end{proof}

\begin{cor} Let $\ct \hh$ denote the tangent Lie algebra via the adjoint representation. If either the almost complex structure $J$ defined by Equation \eqref{defJ} or the almost paracomplex structure $E$ defined in Equation \eqref{defE} is integrable, then $\hh$ is nilpotent.
\end{cor}

\begin{proof} In view of Theorem~\ref{intJE}, the integrability condition for either $J$ or $E$ on $\ct \hh$ (see Example \ref{tan-co}) is given by
$$0= -[y, jx] + [x, jy] - j[x,y]$$
which means that $j$ is a non-singular derivation of $\hh$. And the existence of such a non-singular derivation implies that $\hh$ is nilpotent \cite{Ja}.
\end{proof}

\begin{exa} \label{symplectic} Let $\hh$ denote a Lie algebra with a non-degenerate skew-symmetric bilinear form $\omega$. This $\omega$ induces a linear isomorphism (denoted also by $\omega$) from $\hh  \to \hh^*$ by taking  $\omega(x)(y)=\omega(x,y)$ for $x,y\in \hh$. 
Moreover   $\omega$ gives a symplectic structure on $\hh$ if $d\omega=0$, that is,
$$\omega([x,y],z)+\omega([y,z],x)+\omega([z,x],y)=0\qquad \mbox{ for all }x,y,z\in \hh,$$
which is equivalent to requiring that the map $\omega:\hh\to \hh^*$ is a $1$-cocycle of $(\hh, \ad^*)$. 

Conversely, a non-singular $1$-cocycle $j$ of $(\hh, \ad^*)$ gives a symplectic structure on $\hh$ when $\omega(x,y):= j(x)(y)$ is skew-symmetric. 
\end{exa}

\begin{exa} The cotangent Lie group is naturally equipped with a left-invariant neutral metric.  
A left-invariant  Hermitian structure (for the neutral metric)    on  the  cotangent  Lie  group $T^*G$ is  given  by linear morphisms $J$
of $\ct^*\ggo$, whose  matrix  form  with  respect  to  the  decomposition $\ggo\oplus \ggo^*$ is
$$ 
J=\left( \begin{matrix}  J_1 & J_2\\
J_3 & J_4 
\end{matrix} \right)
$$
satisfying
\begin{enumerate}[(i)]
\item $J_4=-J_1^*$, \quad $J_2=-J_2^*$,\quad $J_3=-J_3^*$;
\vspace{4pt}\item $J_1^2+J_2J_3=-\Id$, \quad $J_1 J_2=-(J_1 J_2)^*,\quad J_3 J_1=-(J_3 J_1)^*$;
\vspace{4pt}\item $J$ is integrable.
\end{enumerate}
A complex structure $I$ on $\ggo$ such that $I^*\alpha=\alpha \circ I$ for  $\alpha\in \ggo^*$ induces $J$ as above on $\ct^*\ggo$, by  taking $J_1=I$, $J_4=-I^*$, $J_2=J_3=0$. 

Let $\omega$ denote a sympletic structure on $\ggo$ and consider also the associated linear isomorphism $\omega:\ggo \to \ggo$ as in Example \ref{symplectic}. Then, $\omega$ induces an Hermitian complex structure on $\ct^*\ggo$, by defining $J(x,\alpha)=(-\omega^{-1}(\alpha), \omega(x))$. 

 Observe that left-invariant Hermitian complex structures on $T^*G$ correspond to left-invariant generalized complex structures on $G$ (see for instance  \cite{ABDF}). 
\end{exa}

Recall that two complex structures $J$ and $J'$ on a given Lie algebra $\ggo$ are called {\em equivalent} if there exists an automorphism $\psi:\ggo \to \ggo$, such that $\psi\circ J'=J\circ \psi$.

Fix the splitting $\ggo=\hh\oplus \kk$ on  the semidirect product Lie algebra $\ggo$ and let $\varphi:\ggo \to \ggo$ be a linear map with matrix
$$\varphi = \left( \begin{matrix}
 A & C \\
B & D
\end{matrix} \right).
$$
Let $\ggo=\hh\oplus \kk$ be as above with associated representation $\pi$, $F$ the canonical product structure associated to this decomposition  as in Equation (\ref{defF})  and $\varphi:\ggo \to \ggo$ denote an automorphism. The following statements are equivalent:
\begin{itemize}
\item the subspaces $\hh$ and $\kk$ are $\varphi$-invariant; 
\item  $\varphi \circ F = F \circ \varphi$.
\end{itemize}

Let $\varphi$ denote now an automorphism of $\ggo=\hh \oplus \kk$ preserving the paracomplex structure $F$. Thus,
\begin{itemize}
\item $A:\hh \to \hh$ and $D:\kk \to \kk$ are automorphisms and
\item $\pi(Ax) \circ D = D \circ \pi(x)$ for every $x\in \hh$.
\end{itemize}

Let $J$ and $J'$ be totally real almost complex structures on $\ggo$ associated to the linear isomorphism $j,j':\hh\to \kk$. Let  $\varphi:\ggo \to \ggo $ denote a linear map of $\ggo$ with matrix as above preserving $F$ and satisfing  $J \circ  \varphi  = \varphi \circ J'$. This occurs   if and only if $Dj'x= jA x$ for all $x\in \hh$.

Observe that if such a $\varphi$ is  an automorphism, we  get an equivalence between $J$ and $J'$ which also preserves $F$. Thus, in such a case we also have an equivalence between the almost complex product structures $(F,J)$ and $(F,J')$. By the same argument, we also get equivalences between $(E,J)$ and $(E',J')$ for $E= -F J$ and $E'= -F J'$.

We now turn our attention to some special classes of almost complex and paracomplex structures.
 An almost complex structure $J$ on a Lie algebra $\ggo$ is said to be
\begin{itemize}
\item[c1)] {\em bi-invariant} if $ [JY,Z]  =  J[Y,Z]$ for all $Y,Z$; 
\vspace{4pt}\item[c2)] {\em abelian} if $[JY,JZ] =  [Y,Z]$ for all $Y,Z$;
\vspace{4pt}\item[c3)] {\em anti bi-invariant} if $[JY,Z]  =  -J[Y,Z]$ for all $Y,Z$.
\end{itemize}
An almost complex structure of type c3) is integrable only when $\ggo$ is abelian, while structures of type c1) and c2) are  always integrable. In fact, direct calculations show that the  bi-invariant condition of $J$ implies that the eigenspaces $\ggo_{\pm i}$ are ideals, while in the case of abelian structures c2) one has abelian subalgebras $\ggo_{\pm i}$.

 In particular, for the almost complex structure $J$ on $\ggo=\hh\oplus \kk$  defined in Equation \eqref{defJ}: 

\begin{enumerate}
\item $J$ bi-invariant yields that for all $x,y \in \hh$:
$J[x,jy]  =  [jx,jy]=0 $, since the left side of the equality belongs to $\hh$ while the right side  belongs to $\kk$.

Thus,  $\kk$ is  abelian and $J\pi(x) jy=0$ for all $x,y \in \hh$. Since $j$ is  non-singular,   it  follows that
\begin{enumerate}[a)]
\item $\pi=0$ and
\item $\hh$ is abelian, since for any $x, y\in \hh$, one also has $J[x, Jy]= [x, J^2y]$. 
\end{enumerate}

Therefore, we have the isomorphism  $\ggo \simeq \RR^m \oplus \RR^m \simeq \RR^{2m}$, the abelian Lie algebra  as a direct sum as Lie algebras of two totally real abelian ideals.

\item $J$  abelian means that for $x, y\in \hh$, we get $[jx, jy]=[x,y]=0$, which says that both $\hh$ and $\kk$ are abelian.  Moreover, we must have $\pi(x)jy=\pi(y) jx$ for all $x,y\in \hh$. Thus, $\ggo\simeq \RR^{m} \oplus \RR^{m}$ is a tangent Lie algebra $\ct_{\pi} \RR^m$ via the representation $\pi$.
\end{enumerate} 

\begin{exa} Let $\hh_{n}$ denote the Heisenberg Lie algebra of dimension 2n+1. It is spanned by the vectors $x_1,\hdots, x_n,y_1,\hdots,y_n,z$ satisfying the non-trivial Lie bracket relations $[x_{i}, y_{i}]=z$ for $i=1, \hdots, n$.

The Lie algebra  $\ggo= \hh_{n} \times \RR$ admits an abelian complex structure $J$, defined by
$$J z =e_{0},  \qquad J x_i=y_i ,  \qquad i=1, \hdots n,$$
where $e_{0}$ spans $\mathbb R$ in $\g=\h_{2n+1} \times \RR$.  It is clear that the abelian Lie subalgebras $\aa_1=span\{x_i, z \}_{i=1}^n$ and $\aa_2=span\{y_i , e_0 \}_{i=1}^n$ are totally real with respect to $J$. 
\end{exa}

An almost paracomplex structure $E$ is said to be 
\begin{itemize}
\item[e1)] {\em bi-invariant} if  $E[Y,Z] = [Y,EZ]$; 
\vspace{4pt}
\item[e2)]  {\em abelian} if $[EY,EZ] = -[Y,Z]$. 
\end{itemize}
Again, each  of these conditions implies the integrability of $E$. 
 
\begin{exa}  Making use of  the definition, it is easily seen that the canonical paracomplex structure  on $\ggo=\hh\oplus \kk$ given by $F(x,v)=(x,-v)$ is bi-invariant. 
\end{exa}

With regard to the almost paracomplex structure $E$ on $\ct_{\pi} \hh$ defined in Equation \eqref{defE}, we can  proceed as in the above case for the complex structure $J$. The results we obtain show that the existence of such abelian or bi-invariant structures on the semidirect product Lie algebra $\ggo$ imposes restrictions on the algebraic  structure of $\ggo$.

\begin{prop}\label{specJE} Let $\ggo=\hh\oplus \kk$ denote a semidirect  Lie algebra  and $J,E$ the almost complex and paracomplex structures defined in Equations \eqref{defJ} and \eqref{defE} respectively. Then:

\begin{itemize}
\item For the bi-invariant case, the following statements are equivalent:

\begin{enumerate}[(i)]
\item the almost complex structure $J$ is bi-invariant; 
\item $\ggo$ is abelian;
\item the almost paracomplex structure $E$ is bi-invariant.
\end{enumerate}

\item For the abelian case,
\begin{enumerate}[(i)] 
\item  $J$ is abelian if and only if $\hh$ and $\kk$ are abelian and  $\pi(x)jy=\pi(y) jx$ for all $x,y\in \hh$.
\item $E$ is abelian if and only if $\hh$ and $\kk$ are abelian and $\pi(x)jy=-\pi(y) jx$ for all $x,y\in \hh$.
\end{enumerate}
\end{itemize}
\end{prop}

\begin{rem} Observe that while the canonical paracomplex structure $F$ on $\ggo=\hh\oplus \kk$ is always bi-invariant, by requiring $E$ to be bi-invariant we get that $\ggo$ is necessarily abelian. 
\end{rem}

\section{A parallel connection for the almost complex  product structure}
 In this section we investigate the existence of a torsion free connection  on a semidirect product Lie algebra $\ggo=\hh\oplus \kk$ such that for the almost complex $J$ and almost paracomplex $E$ it holds  $\nabla E=0$ and  $\nabla J=0$.

Let $\nabla$ denote  any connection on a Lie algebra  $\ggo$. The {\em torsion} of $\nabla$ is defined as
$$ T(X,Y)=\nabla_X Y - \nabla_Y X - [X,Y],$$
so that $\nabla$ is said to be {\em torsion-free} is $ T\equiv 0$. The curvature of $\nabla$ is given by the tensor
$$ R(X,Y)Z=[\nabla_X, \nabla_Y] Z- \nabla_{[X,Y]}Z.$$

\noindent
In particular, we  say that the connection    $\nabla$ is {\em flat} if $R\equiv 0$.

Let $P:\ggo \to \ggo$ denote the product structure associated to a semidirect product Lie algebra  $\ggo= \bb \oplus \cc$, where $\bb$ is a Lie subalgebra and $\cc$ is an ideal of $\ggo$. A general element of $\ggo$ can be written as $x+v$ for $x\in \bb$ and $v\in \cc$. Thus, the product structure $P$ is given by
\begin{equation}\label{PP}
P(x + v)=x - v \qquad \mbox{ for all }\quad x\in \bb, v\in \cc.
\end{equation}

Denote by  $p_{\bb}$ and $p_{\cc}$ the projections onto $\bb$ and $\cc$ respectively with respect to the decomposition $\ggo=\bb \oplus \cc$.

Let $\nabla$ denote a connection of $\ggo$ such that $\nabla P\equiv 0$. Then,  one has
\begin{equation}\label{nablaP}
\nabla_{x+ u} (y- v)=P \nabla_{x+ u} (y+v)\quad \text{for all} \; x,y\in \bb, \; u,v \in \cc,
\end{equation}
that is,
$$p_{\bb} \nabla_{x+u} (y-v)= { p_{\bb}}\nabla_{x+u} (y+v), \qquad 
{ p_{\cc}} \nabla_{x+u} (y-v)=- { p_{\cc}} \nabla_{x+u} (y+v),
$$
which gives
\begin{equation}\label{conK}
p_{\bb} \nabla_{x+u} v=0= p_{\cc}\nabla_{x+u} y.
\end{equation}

\begin{exa} Let  $\ggo=\bb\oplus \cc$ denote a semidirect product Lie algebra, with $\bb$ subalgebra and $\cc$ an ideal. Let $\nabla^1$ be a torsion-free connection on $\bb$, $\nabla^2$ a torsion-free connection on $\cc$ and $\rho:\bb\to \End(\cc)$ be the representation  defined by $[x,v]=\rho(x) v$ for $x\in \bb, v\in \cc$. 
Then, the connection on $\ggo$ given by 
$$\nabla_{x+u} (y+v)=\nabla^1_{x} y + \nabla^2_{u} v +\rho(x) v$$ is torsion-free.
\end{exa}

We shall now prove that the above example gives the conditions for $\nabla P\equiv 0$.

\begin{lem}\label{prop11} Let $P$ denote the product structure associated to  the semidirect product Lie algebra $\ggo = \bb\oplus \cc$ as in (\ref{PP}). If $\nabla$ is a torsion-free connection on $\ggo$ such that  $\nabla P\equiv 0$, then it has the form
\begin{equation}\label{tf-K}
\nabla_{x+u}(y+v) = \nabla^1_{x} y + \nabla^2_{u} v + \rho(x) v,
\end{equation}
where $\nabla^1$ is a torsion-free connection on $\bb$, $\nabla^2$ is a torsion-free connection on $\cc$, and $\rho$ is the representation $\rho:\bb\to \End(\cc)$m associated to the splitting $\ggo=\hh\oplus_{\rho} \kk$.
\end{lem}
\begin{proof} By using Equation~(\ref{conK}) and asking the connection $\nabla$ to be torsion-free, one has:
\begin{itemize}
\item $\nabla_x v- \nabla_{v} x =p_{\cc}\nabla_x v - p_{\bb} \nabla_{v} x  = \rho(x) v\quad \text{for all}\; x\in \bb, v\in \cc.$ 

Therefore, $\nabla_{x} v=\pi(x) v\in \cc$ for every $x\in \bb,v\in \cc$ and  $p_{\bb} \nabla_{v} x =0$,  which gives $0= \nabla_{v} x$ for all $x\in\bb, v\in \cc$.

\item $\nabla_x y- \nabla_y x=p_{\bb} \nabla_x y-p_{\bb} \nabla_y x=[x,y]$ for all $x,y\in \bb$, 
that is,  for every $x,y\in\bb$ the bilinear map $(x,y) \to \nabla_x y$ takes values on $\bb$, so it coincides with a torsion-free connection, namely $\nabla^1$ on $\bb$. 
\item Finally, by Equation~(\ref{conK}) one has  $\nabla_{v}w\in \cc$  and this gives a torsion-free connection $\nabla^2$ on $\cc$. 
\end{itemize}
This proves (\ref{tf-K}). 
\end{proof}

Note that the curvature of such a connection is completely described  by
\begin{itemize}
\item $R(x,y)z=R^1(x,y)z$ for all $x,z,z\in\bb$,
\vspace{4pt}\item $R(u,v)w=R^2(u,v)w$ for all $u,v,w\in \cc$,
\vspace{4pt}\item $R(x, v)z=0$ for all $x,z\in \bb, v \in \cc$ and 
\vspace{4pt}\item $R(x,v)w=\pi(x)\nabla^2_{v}w -\nabla^2_{v} \pi(x) w - \nabla^2_{\pi(x)v} w$, for all $x\in \bb, v,w\in \cc$. 
\end{itemize} 

\begin{exa} Let $\ggo=\bb\oplus \cc$ be a semidirect product Lie algebra, where $\cc$ is an ideal with $\dim \bb=\dim\cc$ and assume $\rho:\bb\to \End(\cc)$ is a representation acting by derivations. 
Take a non-singular cocycle of $(\bb,\rho)$. Then the connections on $\bb$ and $\cc$ { respectively given} by
$$\nabla^1_x y=j^{-1} \rho(x) jy\quad \text{and} \quad \nabla^2_v w=\frac12 \ad(v) w \qquad \mbox{for all }x,y\in \bb, v, w\in \cc,$$
 build a torsion-free connection on $\ggo$ such that $\nabla P=0$,  being $P$ the paracomplex structure associated to the splitting $\ggo=\bb\oplus_{\rho} \cc$. 

Indeed, for $\cc$ abelian the connection $\nabla^2$ here would be trivial and the connection $\nabla$ on $\ggo=\bb\oplus_{\rho} \cc$  will be completely determined by $\nabla^1$ and $\rho$. 
\end{exa}

We now investigate the conditions for a connection that parallelizes the almost complex structure $J$ defined in Equation \eqref{defJ}. 
Let $\nabla$ denote  a connection on the semidirect Lie algebra $\ggo=\hh\oplus \kk$, with $\kk$ an ideal such that $\dim \hh=\dim \kk$. By using the definition of $J$, it is easily seen that 
$$
(\nabla _{x_1 + jy_1} J)(x_2 + jy_2)= \nabla_{x_1 + jy_1}(-y_2 + jx_2)- J\nabla _{x_1 + jy_1}(x_2 + jy_2).$$

Hence, we have that $\nabla J=0$ if and only if
\begin{equation}\label{nab1}
p_{\kk}\nabla _{x_1 + jy_1 }(x_2 + jy_2)= - j p_{\hh}\nabla_{x_1 + jy_1}(y_2 - jx_2), \quad \text{for all} \;\, x_i, y_j \in \h. 
\end{equation}
Making use of this and the Lemma \ref{prop11} we {obtain} the  following characterization.

\begin{thm}\label{thm2} Let $J$ denote the almost complex structure on $\ggo=\hh\oplus \kk$ defined in Equation \eqref{defJ} and let $F$ be the paracomplex structure given by $F(x + v)=(x - v)$ for $x,\in \hh, v\in \kk$. 
Then, there exists a torsion-free connection $\nabla$ on $\ggo=\hh\oplus \kk$ such that $\nabla F=0$ and $\nabla J=0$ if and only if $J$ is integrable.

Moreover, such a connection is uniquely determined by 
$$\nabla_{x_1 + j y_1} (x_2 + jy_2)=\tilde{\pi}(x_1)x_2+\pi(x_1)j y_2, \quad \mbox{ with } x_1, x_2, y_1, y_2\in \hh,$$
where $\tilde{\pi}(x_1)\in \End(\hh)$ is given by $\pi(x_1)x_2=j^{-1} \pi(x_1) j x_2$. 
\end{thm}

\begin{proof} Let $\nabla$ denote a torsion-free connection on $\ggo=\hh\oplus \kk$, such that $\nabla F=\nabla J=0$. Thus, the connection   $\nabla$ satisfies the conditions listed in Lemma \ref{prop11}, which  together with  Equation~\eqref{nab1} gives
$$
\begin{array}{rcl}
\nabla_{x_1+jy_1}(-y_2+jx_2) & = & - \nabla^1_{x_1} y_2+ \pi(x_1) jx_2+\nabla^2_{jy_1}jx_2\\[4pt]
& = & J (\nabla^1_{x_1} x_2  +\pi(x_1) jy_2 + \nabla^2_{jy_1}jy_2), 
\end{array}
$$
for the connections $\nabla^1:\hh \times \hh \to \hh$ and $\nabla^2:\kk \times \kk \to \kk$. 
The components on  $\hh$ and $\kk$ respectively of such equations give, for all $x_1, x_2, y_1, y_2\in\hh$:
\begin{itemize}
\item[(a)] on $\hh$: $-\nabla^1_{x_1} y_2=-j^{-1}\pi(x_1) jy_2 -j^{-1}\nabla^2_{jy_1}jy_2$, 
\vspace{4pt}\item[(b)] on $\kk$: $\pi(x_1) jx_2+\nabla^2_{jy_1}jx_2=j \nabla^1_{x_1}x_2$.
\end{itemize}
Observe that since these equations must hold  for arbitrary elements, taking $x_1=0$  we obtain $\nabla^2\equiv 0$. Thus, this implies the following:
\begin{itemize}
\item from (b): $\nabla^1_{x_1}x_2= j^{-1} \pi(x_1) jx_2$;
\vspace{4pt}\item by asking $\nabla^2$ to be  torsion-free, we get that $\kk$ should be abelian. 
\end{itemize}

 For the connection $\nabla^1$, the condition of being  torsion-free is equivalent to the fact that the identity map is a $1$-cocycle of $(\hh, \tilde{\pi})$. In fact, this follows by applying $j^{-1}$ to the equation for $j$ to be a $1$-cocycle of $(\hh, \pi)$. So,  $J$ is integrable.

Conversely, by Theorem~\ref{intJE} the integrability of  $J$ is equivalent to  $\kk$ is abelian and $j$ is a non-singular isomorphism which is a $1$-cocycle of  $(\hh, \pi)$. Thus, the connection $\nabla$ defined by
\begin{equation} \label{tf-KJ}
\nabla_{x_1+jy_1}(x_2+jy_2) = \tilde{\pi}(x_1)  x_2 + \pi(x_1) jy_2,
\end{equation}
where $\tilde{\pi}(x_1):= j^{-1} \pi(x_1)j$, is an equivalent representation of $\pi$ (and also a torsion-free connection on $\hh$), is torsion-free and it satisfies $\nabla F=\nabla J=0$. 
\end{proof}

\begin{rem} Under the assumptions of the above Theorem, since $\pi$ is a representation one gets

$$0=\pi(x)\pi(y) jz-\pi(y)\pi(x)jz-\pi([x,y])jz=j\nabla^1_x\nabla^1_y z- j\nabla^1_y\nabla^1_y z-j\nabla^1_{[x,y]} z,$$

that is the connection  $\nabla^1$ on $\hh$ is flat. Since  $\nabla^2\equiv 0$, one concludes that $\nabla$ is flat. 
\end{rem}

Observe that since $F=E J$ (equivalently, $E=-FJ$, $J=EF$),  if a given connection $\nabla$ parallelizes any pair of structures among $J,E,F$,  it necessarily parallelizes also the remaining one. In particular, { we  see that} for any linear isomorphism $j:\h \to \kk$ we have several statements which are equivalent to  Theorem~\ref{thm2} for the almost complex structure $J$ and the almost paracomplex structures $E$ already defined, whose proof simply follows from the relations among $J,E,F$. 

\vspace{8pt}

{\bf Theorem \ref{thm2}.'} Let $\ggo=\hh\oplus\kk$ denote a semidirect product Lie algebra and let $j:\h \to \kk$ be any linear isomorphism.  Then, the following statements are equivalent:
\begin{itemize}
\item There { exists a torsion-free} connection $\nabla$ on $\ggo$, such that $\nabla F=0$ and $\nabla J=0$;
\item There { exists a torsion-free} connection $\nabla$ on $\ggo$, such that $\nabla F=0$ and $\nabla E=0$ ; 
\item There { exists a torsion-free} connection $\nabla$ on $\ggo$, such that $\nabla E=0$ and $\nabla J=0$;
\item $J$  is integrable; 
\item $E$ is integrable.
\end{itemize}
In any case,   such a connection is unique and flat. See Theorem \ref{thm2}.

\begin{rem}
As proved in \cite{AS}, for a {\em complex product structure} $\{J,E\}$ on a Lie algebra there exists a unique torsion-free connection which
makes both $J$ and $E$ parallel. The above theorem extends this result to the wide class of 
{\em almost} complex product structures  on a semidirect product Lie algebra $\ggo=\hh\oplus \kk$, defined by any linear isomorphism $j: \h \to \kk$. We do not ask $J$ to be integrable,  while $F$ is integrable with one of the eigenspaces as an ideal.
\end{rem}

\section{Affine structures and split isomorphisms}

In this section we study the connection of the previous section as an LSA structure. We start by recalling some generalities of affine structures. 

Let $V$ denote a real vector space. A {\em left-symmetric algebra} structure on $V$ ({\em LSA} for short)  consists of a bilinear map operation $\cdot:V \times V\to V$, whose associator 
$$a(x,y,z)=(x\cdot y)\cdot z- x\cdot (y\cdot z)$$
 satisfies
$a(x,y,z)=a(y,x,z)$. 

LSA are also known as pre-Lie algebras or Koszul-Vinberg algebras.  
An LSA gives rise to a Lie algebra structure on $V$, where the  Lie bracket is  defined by

\begin{equation}\label{bracketA}
[x,y]_L:= x\cdot y-y\cdot x\qquad \mbox{for all }\; x,y\in V.
\end{equation}
We shall refer to this as to the Lie algebra underlying the LSA $(V,\cdot)$. An LSA on the vector space $V$ with associator $a\equiv 0$ is nothing but an associative algebra structure on $V$. 
On the other hand, it is known that Equation~\eqref{bracketA} is the usual way to define a Lie algebra bracket on any associative algebra.

\begin{prop}{\bf \cite{Bu1,DM}} \label{p100}
There is a canonical one-to-one correspondence between the following classes of objects, up to suitable equivalence:
\begin{itemize}
\item $\{${\it 
Left-invariant affine structures on} $G\}$;
\vspace{4pt}\item $\{${\it 
Affine structures on the Lie algebra} $\ggo \}$;
\vspace{4pt}\item $\{$LSA structures on $\ggo \}$. 
\end{itemize}
\end{prop}

An {\em affine structure} on a Lie algebra $ \ggo$ is a connection $\nabla:\ggo \times \ggo \to \ggo$ which is torsion-free and flat. Such  connection gives rise to a left-invariant affine connection on the corresponding Lie group $G$. See more about LSA or affine structures in \cite{AM,Bu1,Bu,Bu2, Se} for instance.

\begin{exa}  Whenever $\kk$ is abelian, the connection $\nabla$ of Theorem \ref{thm2} gives an affine structure on the semidirect product  Lie algebra $\ggo=\hh\oplus \kk$.
\end{exa} 

Let $\nabla$ denote an affine structure on a Lie algebra $\ggo$. By defining 
$$x\cdot y=\nabla_x y$$
one gets an LSA structure $\cdot$ on $\ggo$, such that the Lie bracket induced by the binary operation coincides with the original Lie bracket on $\ggo$. In fact,
$$a(x,y,z)= \nabla_{\nabla_x y}z- \nabla_x \nabla_y z ,$$
so that $a(x,y,z)=a(y,x,z)$ yields the following equality:
$$\nabla_{\nabla_x y}z- \nabla_x \nabla_y z= \nabla_{\nabla_y x}z- \nabla_y \nabla_x z. 
$$
On the other hand, since $\nabla$ is torsion-free, one has
$$\nabla_{\nabla_x y}z-\nabla_{\nabla_y x}z = \nabla_{[x,y]} z= \nabla_x \nabla_y z - \nabla_y \nabla_x z
$$
where the second equality follows from the flatness condition of $\nabla$. Thus, the bilinear map $\cdot:\ggo \times \ggo \to \ggo$ gives an LSA structure if and only if it satisfies
\begin{align}
&{[x,y]}=  x \cdot y - y \cdot x, \label{eqN1} \\[2pt]
&{[x,y]}\cdot z   =   (x\cdot y) \cdot z - ( y \cdot x) \cdot z \label{eqN2}
\end{align}
for all $x,y,z\in \ggo$. 

\begin{exa} \label{con-sim} Since the end of 60's-the early 70's (see for example \cite{Chu}), it is well known that if $\omega$ is a symplectic
structure on a Lie algebra $\ggo$, the formula
$$\omega(\nabla_x y, z)=\omega(y, [x,z])$$
defines an affine structure on $\ggo$.
\end{exa}

 Consider a left-symmetric algebra $\cdot$ on a real vector space $V$ and the Lie algebra $\ggo$ underlying $(V,\cdot)$. The left-multiplication $L$ in $V$ is given by $L(x) y= x\cdot y$,  while the right multiplication is defined analogously by $R(x) y= y\cdot x$. The two conditions above then yield that
\begin{align} 
&{L:\ggo \to \End(\ggo)}   \; \mbox{ is a Lie algebra homomorphism}, \label{eqN21}\\[2pt]
&Id:\ggo\to\ggo  \; \mbox{ is a $1$-cocycle of }(\ggo, L). \label{eqN22} 
\end{align}

We now consider $\aff(\ggo)= \End(\ggo) \oplus\ggo$ with the Lie bracket $[(T, x), (S,y)]= (T S - ST, T y - Sx)$. Let $\alpha$ be the linear map  $\alpha:\ggo \to \aff(\ggo)$ given by
$$\alpha(X)= (L(X),X).$$

\begin{prop} \label{propa} The linear map $L$ satisfies \eqref{eqN21} and \eqref{eqN22} if and only if $\alpha$ is a Lie algebra homomorphism. 
\end{prop}

Moreover, we recall the following result proved in \cite{Bu2}. 

\begin{lem} Let $(V,\cdot)$ be an LSA structure with underlying Lie algebra $\ggo$. Then, $\ggo$ is abelian if and only if $(V,\cdot)$ is associative and commutative.
\end{lem}

 \begin{rem} As we already mentioned in the Introduction, LSA structures are related to affine transformations.  Milnor proved in \cite{Mi} that a connected Lie group acts freely by affine transformations of some $\RR^n$ if and only if it is simply connected and solvable. He stated the following question \cite{Mi}.

\vspace{3pt}

{\bf Milnor's Question.} {\it Does every solvable Lie group
$G$ admit a complete left-invariant affine structure, or equivalently, does the universal covering  group $\tilde{G}$ 
operate simply transitively
by affine transformations of $\RR^k$?}

\vspace{3pt}

 In the conditions of this question,  Auslander \cite{Au} proved that  $G$ is solvable.

Recall that the group of affine transformations of $\RR^k$, denoted by $\Aff(\RR^k)$, is given by 
$$\Aff(\RR^k)=\left\{ \left(\begin{matrix} A & b \\
0 & 1 \end{matrix} \right) \quad A\in \GL(\RR^k), b\in \RR^k\right\}
$$
and acts on a vector {\small${\left(\begin{matrix} v \\ 1 \end{matrix} \right)}$}$\in \RR^k$ by
$$\left(\begin{matrix} A & b \\
0 & 1 \end{matrix} \right) \left(\begin{matrix} v \\ 1 \end{matrix} \right)  = \left(\begin{matrix} A v +b \\
1 \end{matrix} \right).
$$
Observe that the Lie algebra of $\Aff(\RR^k)$ is $\aff(\RR^k)\subset \End(\RR^k)$.
 
\end{rem}

 If we start from a Lie algebra $(\hh, [\,,\,])$, an LSA structure $\cdot$ on $\h$ is said to be {\em compatible} if the Lie bracket defined by $\cdot$, as described in \eqref{bracketA},  coincides with the Lie bracket on $\hh$: $[\,,\,]=[\,,\,]_L$.

\begin{exa} \label{exA} An LSA structure $\cdot$ on a real vector space $V$ gives rise to a semidirect product Lie algebra in the following way. Let $\hh$ denote the Lie algebra underlying $(V,\cdot)$. We take the direct sum of vector spaces $\ggo = \hh \oplus V$ with the Lie bracket
$$[(x, u),(y, v)]= ([x,y]_L,  \theta L(x) \theta^{-1} v- \theta L(y)\theta^{-1} u ),
$$
where $[x,y]_L$ is given by \eqref{bracketA}  and  $\theta:\hh \to V$ is a linear bijection. So, $\tilde{L}(x) (v):= \theta L(x) \theta^{-1}(v)$ induces an equivalent representation of $\hh$ to $\End(V)$. It is clear that $\hh$ and $V$ are subalgebras of $\ggo$, $V$ is an abelian ideal of $\ggo$ and  
$\theta:\hh \to V$ is a linear isomorphism which is a $1$-cocycle of $(\hh, \tilde{L})$. 
\end{exa}

\begin{prop} Let $\cdot$ denote an LSA structure on a vector space $V$ and $\hh$ the underlying Lie algebra. Then the canonical inclusion $\iota:\hh \to \ggo=\hh \oplus V$ is a Lie algebra monomorphism, where $\ggo$ is equipped with the Lie bracket of {\em 
Example~\ref{exA}}.
\end{prop}

 Observe that in the context above we have an almost complex {structure} and an almost paracomplex structure on $\ggo=\hh \oplus V$ given respectively by
$$J(x,v)=(-\theta^{-1} v, \theta x) \qquad \qquad E(x,v)=(\theta^{-1} v, \theta x).$$
The natural question { concerning the converse construction is answered by the following result.} 

\begin{thm} Let $\ggo=\hh\oplus \kk$ denote a semidirect product Lie algebra attached to $(\hh,\pi)$. Let $j : \hh\to \kk$ denote a linear isomorphism. Then, the following statements are equivalent:
\begin{enumerate}[(i)]
\vspace{2pt}\item $\nabla_x y:= j^{-1}\pi(x) jy$ is an affine connection on $\hh$;
\vspace{4pt}\item the almost complex structure $J$ defined by $J(x,v)=(-j^{-1}v, jx)$ is integrable. 
\vspace{4pt}\item the almost para-complex structure $E$ defined by $J(x,v)=(j^{-1}v, jx)$ is integrable.
\end{enumerate}
\end{thm}
\begin{proof} We know from Theorem~\ref{thm2}  that the integrability of $J$ on $\ggo=\hh\oplus\kk$ is equivalent to the fact that $\kk$ is an abelian ideal and $j:\hh\to\kk$ is a 
$1$-cocycle of $(\hh, \pi)$. Thus, $j^{-1} \circ j= Id:\hh \to \hh$ is a $1$-cocycle of the equivalent representation  $\tilde{\pi}(x)=j^{-1}\pi(x)j$, and the conclusion follows from Proposition~\ref{propa}. In fact, $\tilde{\pi}:\hh\to\End(\hh)$ is a torsion-free flat connection on $\hh$.
The converse also follows from  Theorem \ref{thm2}.
\end{proof}

The results proved throughout this work now yield the following main Theorem. More precisely, denoting by $J,E,F$ the structures described in 
Section~2, we have the following result.

 \begin{thm}\label{main} Let $\ggo=\hh\oplus \kk$ denote a semidirect product Lie algebra with $\dim \hh=\dim \kk$. There are correspondences between the following sets: 
\begin{itemize}
\item Affine structures on $ \hh$;
\vspace{4pt}\item Complex structures $J$ on $\ggo$ for which $\hh$ and $\kk$ are totally real subspaces. 
\vspace{4pt}\item Paracomplex structures $E$ on $\ggo$ such that $E\hh=\kk$; 
\vspace{4pt}\item Torsion-free connections on $\ggo$ parallelizing both the almost complex structure $J$ and the almost paracomplex structure $E$.
\end{itemize}
\end{thm}

\begin{rem} Note that starting with an affine structure $\nabla^1$  on $\hh$ we have an almost complex structure on the semidirect product Lie algebra $\ggo=\hh \oplus_{\pi} V$. Moreover, the connection $\nabla$ on $\ggo$ given by $\nabla_{(x_1, jy_1)} (x_2, jy_2)=(\tilde{\pi}(x_1) x_2, \pi(x_1) y_2)$ is also an affine structure on $\ggo$. Therefore, $\nabla$ permits to construct a complex structure and a para-complex structure on the semidirect product via $\nabla$, $\ggo \oplus V_{\ggo}$, where $V_{\ggo}$ denotes the underlying vector space of $\ggo$ seen as an abelian Lie algebra.
\end{rem}

\begin{exa} {\bf On symplectic structures} 
As we already  mentioned, a left-invariant symplectic structure on a Lie group $H$ induces an Hermitian complex structure on $\ct^*\hh$, which in turn corresponds to some left-invariant generalized complex structures on $G$. This can be read as a one-cocycle of $(\hh,\ad^*)$. 

As  said in Example \ref{con-sim} a symplectic structure $\omega$ on a Lie algebra $\hh$ defines a torsion flat connection $\nabla$. Once we apply the equivalence above for the coadjoint representation, we get exactly this  connection.
\end{exa}

\section{Examples and applications}

In this section we show some applications once we make a natural choice of metri. Explicit computations of complex structrures and affine structures are also done. 

\subsection{Almost Hermitian and para-Hermitian structures on $\ct_{\pi} \hh$}

We now start from any inner product $\la\,,\,\ra$ on the Lie algebra $\hh$ and observe that $\la\,,\,\ra$ can be naturally extended to the semidirect product Lie algebra  $\ggo=\hh\oplus \kk$: 
\begin{itemize}
\item to an inner product $g$:
\begin{equation}\label{gofj}
g((x,u), (y, v)):=\la x, y \ra + \la j^{-1} u, j^{-1} v\ra ;
\end{equation}
\item to some neutral metrics $\bar g$ and $\tilde g$:
\begin{equation}\label{gofE}
 \bar g((x,u), (y, v)):=\la x, y \ra - \la j^{-1} u, j^{-^1} v\ra ;
\end{equation}
\begin{equation}\label{gofF}
 \tilde g ((x,u), (y, v)):=\la x, j^{-1} v \ra + \la j^{-1} u, y \ra ;
\end{equation}
\end{itemize}
It is easy to check that for all $x, y\in \hh$ and $u,v \in \kk$:
\begin{itemize}
\item $g(J(x,u), J(y, v))= g((x,u), (y, v))$;
\vspace{4pt}\item $\bar g(E(x,u), E(y, v))= -\bar g((x,u), (y, v))$; 
\vspace{4pt}\item $\tilde g (F(x,u), F(y, v))= -\tilde g((x,u), (y, v))$.
\end{itemize}
Thus, the pair $(J,g)$ defines an almost Hermitian structure on  $\ggo=\hh\oplus \kk$, while $(E,\bar g)$, $(F,\tilde g)$ give almost para-Hermitian structures on  $\ggo=\hh\oplus \kk$. Observe that $j:\hh \to \kk$ is an isometry with respect to the inner products induced by $g$ on these subspaces.

In view of the equalities above, we have the corresponding $2$-forms:
$$\omega_J(X,Y)=g(JX,Y),\qquad \omega_E (X,Y)=\bar g(EX,Y),\qquad \omega_F (X,Y)=\tilde g (FX,Y).$$
It is  known that $\omega_J$ is called the K\"ahler or fundamental  {$2$-form} of $(J,g)$ and correspondingly, $\omega_E$ and $\omega_F$ are the fundamental {$2$-forms} of the almost para-Hermitian structures $(E,\bar g)$ and $(F,\tilde g)$ respectively.  We may refer to \cite{CFG} for further information concerning almost para-Hermitian structures and related notions. We recall that:
\begin{itemize}
\item  the pair $(J,g)$ is called
{\em almost K\"ahler} if the fundamental form $\omega$ is closed;
\vspace{2pt}\item the pair $(E, \omega_E)$ (respectively,  $(F,\omega_F)$) is called {\em almost para-K\"ahler} if $d\omega_E=0$ (respectively,  
$d \omega_F=0$).  
\end{itemize}
Observe that notwithstanding  the differences under these metrics and structures, we have
\begin{equation}\label{omegas}
\omega_J(X,Y)=-\omega_E(X,Y)=\omega_F(X,Y) \quad \mbox{ for all }X,Y\in \ct_{\pi}\hh.
\end{equation}
Moreover, these 2-forms are non-degenerate. Therefore, they are symplectic if and only if they are closed. We now characterize the closeness condition for these structures.

\begin{prop}  For any linear isomorphism $j: \hh \to \kk$ and  metric $\la\,,\,\ra$ on $\h$, let $(J,g)$ denote the almost Hermitian structure  on the semidirect product Lie algebra $\ggo=\hh \oplus \kk$ described by $J$ as in Equation \eqref{defJ}, and $g$ as above \eqref{gofj}, and let $(E,\bar g)$, $(F,\tilde g)$be  the almost para-Hermitian structures described respectively by Equations \eqref{defE},\eqref{gofE} and by Equations \eqref{defF},\eqref{gofF}. Then, the following properties are equivalent: 
\begin{enumerate}[(a)]
\item $(J,g)$ is almost K\"ahler; 
\item $(E,\bar g)$ is almost para-K\"ahler;
\item $(F,\tilde g)$ is para-K\"ahler;
\item $\kk$ is abelian and 
\begin{equation}\label{equ-almostc}
\la j^{-1} \pi(y) j z, x\ra - \la j^{-1} \pi(x) j z, y\ra=\la [x,y], z \ra \quad \mbox{ for all } x,y, z\in \hh.
\end{equation}
\end{enumerate}
\end{prop}
\begin{proof} The equivalence of (a),(b) and (c) follows at once from the definitions of the two-forms in  \eqref{omegas} and the respective closeness conditions. We now prove the equivalence of (a) and (d). We first observe that because of \eqref{gofj}, $\hh$ is ortogonal to $\kk$ with respect to $g$. Therefore,  $\kk$ is an  isotropic ideal for $\omega_J$ (that is,  $\omega_{|\kk \times \kk}=0$). Hence, if $\omega_J$ is closed (and so, symplectic), then $\kk$ must be abelian. 

Assume that $\kk$ is abelian. By taking $x,y,z\in \hh$, one has
\begin{itemize}
\item  $d\omega_J(x,y,z)= {d\omega_J(jx,jy,jz)=} d\omega_J(jx,y,jz)= 0$ is always satisfied. 
\vspace{4pt}\item $d\omega(x,y,jz)=0$ if and only if Equation \eqref{equ-almostc} holds. 
\end{itemize}

Conversely, it is easily seen by direct calculation  that if \eqref{equ-almostc} holds and $\kk$ is abelian, then $d\omega_J=0$. 

\end{proof}

\begin{cor} Let $\ggo=\hh\oplus \kk$ denote a semidirect product Lie algebra attached to the representation $\pi:\hh\to \End(\kk)$. Let $\la\,,\,\ra$ denote an inner product on $\hh$ such that $j^{-1}\pi(x)j:=\tilde{\pi}(x) \in \sso(\hh)$ for every $x\in \hh$ and $j:\hh\to \kk$ a linear isomorphism. Let $g$ denote the metric on $\ggo$ as above. If $(g, J)$ is almost K\"ahler, then  $(g,J)$ is K\"ahler.
\end{cor}
\begin{proof} Indeed, $\tilde{\pi}(x)= j^{-1} \pi(x) j$ defines an equivalent representation of $\pi$, such that $\tilde{\pi}(x)\in\sso(\hh)$ for all $x\in \hh$. By the previous proposition, $(J,g)$ is almost K\"ahler if and only if  
\begin{itemize}
\item $\kk$ is abelian and 
\item $\tilde{\pi}(x) y - \tilde{\pi}(y) x= [x,y]$ for all $x,y\in \hh$,  which follows from the skew-symmetric property of $\tilde{\pi}$. Therefore,  the linear map ${\Id:}\hh \to \hh$ is a 1-cocycle of $(\hh, \tilde{\pi})$. 
\end{itemize}
This yields that $j: \hh \to \kk$ is a 1-cocycle of $(\hh, \pi)$ and the almost complex structure $J$ given by $J(x, jy)=(-y, jx)$ is integrable. It is not hard to see that the connection $\nabla$ given by $\nabla_{(x_1, jy_1)} (x_2, jy_2)=(\tilde{\pi}(x_1)x_2, \pi(x_1)jy_2)$ is torsion-free and $\nabla g=0$ and so, $(g,J)$ is K\"ahler. 
\end{proof}

We shall now describe some explicit examples  of the equivalences obtained in Theorem~\ref{main} between compatible LSA on a Lie algebra $\h$ and totally real complex structures on a semidirect product Lie algebra $\g =\hh \oplus V$.

\subsection{From $LSA$ to totally real complex structures}

We shall consider some real Lie algebras, equipped with some LSA deduced from the examples described in  \cite{Bu}, 
and construct the corresponding semidirect product Lie algebras and totally real complex structures on them.

Before starting, we make the following general remark.  Let $\cdot$ denote an LSA structure on a vector space $V$ and $\hh$ the underlying Lie algebra. As we explained in the previous section, {\em any} linear isomorphism $j: \h \to V$ determines

\begin{itemize}
\item[a)] a semidirect product Lie algebra $\g=\h \oplus V$, with Lie bracket described by  \eqref{bracketA};
\vspace{4pt}\item[b)] a totally real complex structure $J$ on $\g$, described as in \eqref{defJ}.
\end{itemize}
Thus, starting from a basis $\{e_i\}$ of $\h$, we can consider the corresponding basis $	\{v_i:=j (e_i)\}$ of $V$, with the obvious advantage that with respect to the basis $\{e_i, v_i \}$ of $\g$, the complex structure $J$, defined as in Equation \eqref{defJ}, will be completely { determined by the simple equations}
$$J e_i= v_i, \quad i=1,\dots ,n .$$
Clearly, $V$ being an abelian ideal of $\g$, in any case we shall have $[v_i,v_k]=0$, and the description of the Lie algebra $\g$ will be completed by calculating 
$$[e_i,e_k]=e_i \cdot e_k -e_k \cdot e_i, \quad    [e_i,v_k]=j L(e_i)j^{-1} v_k, \qquad i,k=1,\dots ,n.$$

\medskip
{{\bf Three-dimensional examples}.} In the next subsection we shall illustrate compatible LSA on three-dimensional Lie algebras, whose either tangent or cotangent Lie algebras admit some totally real complex structures. Here, we consider a family of three-dimensional Lie algebras $\h$, which, by the results proved in \cite{CO}, do not admit totally real complex structure neither on $T \h$ nor on $T^* \h$.

For any real constant $\lambda \neq 0, -1<\lambda <1$, consider the three-dimensional Lie algebra $\h =\mathfrak{r}_{3,\lambda}$=span$\{e_1,e_2,e_3\}$, described by 
\begin{equation}\label{Lie3D}
[e_1, e_2]=e_2, \quad [e_1, e_3]=\lambda e_3,
\end{equation}
and on $\h$ the compatible LSA
\begin{equation}\label{LSA3D}
e_1 \cdot e_1 =(\lambda+1) e_1, \quad e_1 \cdot e_2 = e_2, \quad e_1 \cdot e_3 =\lambda e_3, \quad e_2 \cdot e_3 = e_1, \quad e_3 \cdot e_2 = e_1.
\end{equation}
Denoted by $V$ the vector space underlying $\h$ and by $j: \h \to V$ a linear isomorphism, we consider on $V$ the basis $\{v_i := j(e_i), i=1,2,3\}$. We then apply \eqref{bracketA} to $e_i,v_i$ and prove the following.

\begin{prop}\label{ex3D}
For any real constant $\lambda \neq 0, -1<\lambda <1$, consider the three-dimensional real Lie algebra $\h=\mathfrak{r}_{3,\lambda}$ given by \eqref{Lie3D}, the vector space $V$ underlying $\h$ and on it the compatible LSA structure given by \eqref{LSA3D}. Then, the Lie brackets
$$\begin{array}{lllll} 
\left[e_1,e_2 \right]=e_2, \; &    [e_1,e_3]=\lambda e_3, \; &  \\[4pt]
\left[e_1, v_1 \right]=(\lambda+1) v_1 , \; &  [e_1,v_2]=v_2, \; &  [e_1,v_3]=\lambda v_3, \; & 
\left[e_2,v_3 \right]=v_1, \; &  [e_3,v_2]=v_1 \;  \\[4pt]
\end{array}$$
describe a semidirect product Lie algebra $\g = \h \oplus V=span(e_i) \oplus span (v_i)$, admitting the totally real complex structure $J$ completely determined by $Je_i=v_i$.
\end{prop}

For the Lie algebra $\h =\mathfrak{r}_{3,\frac{1}{2}}$, given by \eqref{Lie3D} with $\lambda=\frac 12$, we can also consider a different compatible LSA, namely, 
\begin{equation}\label{LSA3D2}
\begin{array}{l}
e_1 \cdot e_1 =\frac 32 e_1, \;\, e_1 \cdot e_2 = e_2, \;\, e_1 \cdot e_3 =\frac 12 e_3, \;\, e_2 \cdot e_3 = e_1, \;\, e_3 \cdot e_2 = e_1, \;\, e_3 \cdot e_3 = -e_2.
\end{array}
\end{equation}
As before, we denote by $V$ the vector space underlying $\h$ and by $j: \h \to V$ a linear isomorphism and fix on $V$ the basis $\{v_i := j(e_i), i=1,2,3\}$. Applying \eqref{bracketA} to $e_i,v_i$, we obtain the following.

\begin{prop}
Consider the three-dimensional real Lie algebra $\h=\mathfrak{r}_{3,\frac{1}{2}}$ described by \eqref{Lie3D} with $\lambda =\frac 12$, the vector space $V$ underlying $\h$ and on it the compatible LSA structure given by \eqref{LSA3D2}. Then, the Lie brackets
$$\begin{array}{lll} 
\left[e_1,e_2 \right]=e_2, \; &    [e_1,e_3]=\frac 12 e_3, \; &  \\[4pt]
\left[e_1, v_1 \right]=\frac 32 v_1 , \; &  [e_1,v_2]=v_2, \; &  [e_1,v_3]=\frac 12 v_3, \;  \\[4pt]
\left[e_2,v_3 \right]=v_3, \; &  [e_3,v_2]=v_1, \; &  [e_3,v_3]=-v_2  \\[4pt]
\end{array}$$
determine a semidirect product Lie algebra $\g = \h \oplus V=span(e_i) \oplus span (v_i)$, not isomorphic to the one described (for $\lambda =\frac 12$) in the above Proposition~{\em\ref{ex3D}}, admitting the totally real complex structure $J$ given by $Je_i=v_i$.
\end{prop}

\medskip
{{\bf Examples in arbitrary dimension}.}

\medskip
{\bf 1) ${\bf \h=\mathcal{I}_n}$.} 
 Consider the $n$-dimensional Lie algebra $\h =$span$\{e_1,\dots ,e_n\}$, described by 
\begin{equation}\label{LieIn}
[e_1, e_k]=e_k, \quad k \geq 2,
\end{equation}
and on $\h$ the compatible LSA
\begin{equation}\label{LSAIn}
e_1 \cdot e_1 =2 e_1, \quad e_1 \cdot e_k = e_k, \quad e_k \cdot e_k = e_1, \quad k \geq 2.
\end{equation}
Let $V$ denote the vector space underlying $\h$ and $j: \h \to V$ a linear isomorphism. We fix on $V$ the basis $\{v_i := j(e_i), i=1,\dots ,n\}$. Then, applying \eqref{bracketA} to $e_i,v_i$, we prove the following.

\begin{prop}
Consider the $n$-dimensional real Lie algebra $\h=\mathcal{I}_n$ described by \eqref{LieIn}, $V$ the vector space underlying $\h$ and on it the compatible LSA structure given by \eqref{LSAIn}. Then, the Lie brackets
$$[e_1,e_k]=e_k, \quad [e_1,v_1]=2v_1, \quad [e_k,v_k]=v_1, \quad k \geq 2$$
describe a semidirect product Lie algebra $\g = \h \oplus V=span(e_i) \oplus span (v_i)$, admitting the totally real complex structure $J$ completely determined by $Je_i=v_i$.
\end{prop}

\medskip
{\bf 2) ${\bf \h=\mathcal{A}_n}$.} We now consider the $n$-dimensional Lie algebra $\h =$span$\{e_1,\dots ,e_n\}$, described by 
\begin{equation}\label{LieAn}
[e_1, e_k]=\alpha_k e_k, \quad k \geq 2,
\end{equation}
for some real constants $\alpha_1,\dots ,\alpha_n $ satisfying $\alpha_i \neq 0$ for all $i$ and $\alpha_{n+2-k}=\alpha_1-\alpha_k$ for all $k \geq 2$. On $\h$ one has the compatible LSA
\begin{equation}\label{LSAAn}
e_1 \cdot e_i =\alpha_i e_1, \quad e_k \cdot e_{n+2-k} = e_1,  \quad i \geq 1 \; \text{and} \; k \geq 2.
\end{equation}
Denoted by $V$ the vector space underlying $\h$ and by $j: \h \to V$ a linear isomorphism, we consider on $V$ the basis $\{v_i := j(e_i), i=1,\dots ,n\}$ and apply \eqref{bracketA} to $e_i,v_i$, proving the following.

\begin{prop}
For any real constants $\alpha_1,\dots ,\alpha_n $ satisfying $\alpha_i \neq 0$ for all $i$ and $\alpha_{n+2-k}=\alpha_1-\alpha_k$ for all $k \geq 2$, consider the $n$-dimensional real Lie algebra $\h=\mathcal{A}_n$ described by \eqref{LieAn}, the vector space $V$ underlying $\h$ and on it the compatible LSA structure given by \eqref{LSAAn}. Then, the Lie brackets
$$[e_1,e_k]=\alpha_k e_k, \quad [e_1,v_k]=\alpha_k v_k, \quad [e_k,v_{n+2-k}]=v_1, \quad k \geq 2$$
describe a semidirect product Lie algebra $\g = \h \oplus V=span(e_i) \oplus span (v_i)$, admitting the totally real complex structure $J$ completely determined by $Je_i=v_i$.
\end{prop}

\subsection{From totally real complex structures to $LSA$}

Let  $\g = \h \oplus V$ denote  a semidirect product Lie algebra, determined by a representation $\pi$, where $V$ is an abelian ideal. We consider an  isomorphism $j: \h \to V$, which is a cocycle of $(\h,\pi)$. Then, by Theorem~\ref{intJE}, $j$ defines a totally real complex structure $J$ given by Equation \eqref{defJ}. Moreover, by Theorem~\ref{main}, the Lie algebra $\h$ will admit a compatible LSA, {explicitly} given  by
$$x\cdot y = j^{-1} \pi(x)j y, 	\quad x,y \in \h.$$
We shall now start from some examples of totally real complex structures obtained in \cite{CO} on tangent and cotangent Lie algebras, and  describe the corresponding compatible LSA.

\medskip

{\bf Tangent Lie algebras.} As proved in \cite{CO}, totally real complex structures on the tangent Lie algebra $T \h$ of $\h$ can only occur when $\h$ is nilpotent. In particular, if $\dim \h =3$, then the only possibility is given by $\h=\h_1$, where $\h_1$ denotes the three-dimensional Heisenberg Lie algebra. Following \cite{CO}, the tangent Lie algebra $T \h _1$:
$$\begin{array}{lll}
\left[ e_1,e_2 \right]=e_3, \; & \left[ e_1,v_2 \right]=v_3, \; & \left[ e_2,v_1 \right]=-v_3
\end{array}$$
admits the totally real complex structures
\begin{equation}\label{Jh1}
J_s e_1=v_1, \quad J_s e_2=-s v_1+v_2, \quad J_s e_3=2v_3, \qquad s=0,1. 
\end{equation}
Clearly, by Equation \eqref{defJ}, these totally real complex structures correspond to the linear isomorphisms $j_s$ completely determined by 
$j_s e_i=J_s e_i$, $i=1,2,3$. Then, on $\h_1$ we have the corresponding compatible LSA given by
$$e_i \cdot _s e_j = j_s ^{-1} \ad _{e_i} j_s (e_j), \quad i,j=1,2,3,$$
determined by $j_s$ and the adjoint representation. Applying the equation above, we obtain the following.

\begin{prop}
For $s=0,1$, consider the totally real complex structures $J_s$ on the tangent Lie algebra of $\h_1$ described in \eqref{Jh1}. Then, 
$\h_1$ admits the corresponding compatible LSA structures $\cdot _s$, described by 
$$\begin{array}{l} 
e_1 \cdot _s e_2 =\frac 12 e_3,  \quad  e_2 \cdot _s e_1 =-\frac 12 e_3,  \quad e_2 \cdot _s e_2 =\frac s2 e_3 .  
\end{array}$$
\end{prop}

Let now consider again the $(2n+1)-$dimensional Heisenberg Lie algebra $\h_n$, described by $\left[x_i,y_i \right]= z$. 
On the tangent Lie algebra $T \h _n$ there exist many totally real complex structures, because they are in a one-to-one correspondence with the nonsingular derivations of $\h_n$ \cite{CO}. For example, if we take the totally real complex structure corresponding to the very simple nonsingular derivation $j:\h \to \h$ described by
$$j x_i=x_i, \quad j y_k = y_k, \quad j z = 2z,$$
then we have the corresponding compatible LSA on $\h_n$, described by
$$\begin{array}{l}
x_i \cdot y_k = \frac 12 \delta_{ik} z, \quad i,k=1,\dots,n.
\end{array}$$
Further compatible LSA on $\h_n$ can be easily constructed by the same argument.

\medskip
{\bf Cotangent Lie algebras.} Totally real complex structures on cotangent Lie algebras $T^* \h$ of three-dimensional
Lie algebras  $\h$ were completely classified in \cite[Proposition 3.8]{CO}.  They exist if and only if either $\h$ is unimodular or isomorphic to $\mathbb{R} \times \mathfrak{aff}(\mathbb R)$. In these cases, the linear isomorphism $j:\h \to \h ^*$, determining the complex structure by Equation \eqref{defJ}, admits a matrix representation as follows:
$$\begin{array}{ll}
T^* \h_1: \; \left(\begin{array}{ccc}
a & b & c \\
d & e & f \\
-c & -f & 0
\end{array}\right);  & \qquad
T^* \mathfrak{r}_{3,-1}: \; \left(\begin{array}{ccc}
a & b & c \\
-b & 0 & d \\
-c & -d & 0
\end{array}\right); \\ \\  
T^*  \mathfrak{r}_{3,0}: \; \left(\begin{array}{ccc}
a & b & c \\
-b & 0 & 0 \\
d & 0 & e
\end{array}\right);  & \qquad
T^* \mathfrak{r}'_{3,0}: \; \left(\begin{array}{ccc}
a & b & c \\
-b & 0 & d \\
-c & -d & 0
\end{array}\right);
\end{array}$$
for some real constants $a,\dots,f$. Observe that $\mathfrak{r} _{3,0}=\mathbb{R} \times \mathfrak{aff}(\mathbb R)$. Because there exist totally real complex structures for the cotangent Lie algebras $T^* \mathfrak{r} _{3,\lambda}$ when $\lambda=0,-1$, in the previous subsection we gave new examples of totally real
complex structures for semidirect product Lie algebras $\mathfrak{r} _{3,\lambda} \oplus V$ assuming $\lambda \neq 0$ and $-1<\lambda <1$.

In the case of $T^* \mathfrak{r} _{3,-1}$ we shall now describe the compatible LSA corresponding to the whole family of totally real complex structures described above. The tangent Lie algebra $T^* \mathfrak{r} _{3,-1}$ is given by (see \cite{CO})
$$\begin{array}{llll}
\left[ e_1,e_2 \right]=e_2, \; & \left[ e_1,e_3 \right]=-e_3, \; & \; & \\[4pt]
\left[ e_1,v_2 \right]=-v_2, \; & \left[ e_1,v_3 \right]=v_3, \; &  \left[ e_2,v_2 \right]=v_1, \; & \left[ e_3,v_3 \right]=-v_1  \\
\end{array}$$
and admits a four-parameter family of totally real complex structures, namely, the ones determined by the above linear isomorphisms $j$, for all values of real constants $a,b,c,d$, such that $\det (j)=ad^2 \neq 0$. Then, each of these isomorphisms determines on $\mathfrak{r} _{3,-1}$ a corresponding compatible LSA, explicitly given by
\begin{equation}\label{coadj}
e_i \cdot  e_j = j ^{-1} \ad^* _{e_i} j (e_j), \quad i,j=1,2,3,
\end{equation}
determined by $j$ and the coadjoint representation. Applying \eqref{coadj}, we then prove the following.

\begin{prop}
Consider the totally real complex structures $J$ on the cotangent Lie algebra of $\mathfrak{r} _{3,-1}$, corresponding to the linear isomorphisms $j: \h \to \h^*$ described above, for any real constants $a,b,c,d$ with $ad \neq 0$. Then, 
$\mathfrak{r} _{3,-1}$ admits the following corresponding compatible LSA structures: 
$$\begin{array}{ll} 
e_1 \cdot e_1 =\frac bd  e_3,  \quad & e_1 \cdot e_2 =\frac{1}{ad} (-bd e_1+ (ad+bc) e_2-b^2 e_3),  \\[4pt]
e_1 \cdot e_3 =\frac{1}{ad} (cd e_1-c^2 e_2-(ad-bc) e_3),   \quad &
e_2 \cdot e_1 =-\frac{b}{ad} (d e_1- c e_2+b e_3),  \\[4pt]
e_2 \cdot e_3 =\frac{1}{a} (d e_1-c e_2+b e_3),  \quad &
e_3 \cdot e_1 =\frac{c}{ad} (d e_1-c e_2+b e_3),  \\[4pt]
e_3 \cdot e_2 =\frac{1}{a} (d e_1-c e_2+b e_3). 
\end{array}$$
\end{prop}

For the remaining cases, namely, $T^* \h_1$, $T^*  \mathfrak{r}_{3,0}$ and 
$T^*  \mathfrak{r}'_{3,0}$, we shall restrict ourselves to some examples of totally real complex structures explicitly described in \cite{CO}, leaving the calculation of further compatible LSA to the interested reader. Since we use the same argument already explained in the case of $T^*  \mathfrak{r}_{3,-1}$, we only report below, for the different cases, the description of the cotangent Lie algebras and of the totally complex structures we are considering.
$$\begin{array}{llll}
T^* \h_1: \; & \left[ e_1,e_2 \right]=e_3, \; & \left[ e_1,v_3 \right]=-v_2, \; &  \left[ e_2,v_3 \right]=v_1,  \\[4pt]
& J e_1 = v_1, & J e_2 = v_3, & J e_3 =-v_2; \\[6pt]
T^*  \mathfrak{r}_{3,0}: \; & \left[ e_1,e_2 \right]=e_2, \; & \left[ e_1,v_2 \right]=-v_2, \; &  \left[ e_2,v_2 \right]=v_1,  \\[4pt]
& J e_1 = v_2, & J e_2 = -v_1, & J e_3 =v_6; \\[6pt]
T^* \mathfrak{r}'_{3,0}: \; & \left[ e_1,e_2 \right]=-e_3, \; & \left[ e_1,e_3 \right]=e_2, \; &  \left[ e_1,v_2 \right]=-v_3,  \\[4pt]
& \left[ e_1,v_3 \right]=-v_2, \; & \left[ e_2,v_3 \right]=-v_1, \; &  \left[ e_3,v_2 \right]=v_1,  \\[4pt]
& J e_1 = \varepsilon v_1=\pm e_1, & J e_2 = v_3, & J e_3 =-v_2 .
\end{array}$$
Applying Equation \eqref{coadj} to the structures above, we prove the  following.

\begin{prop}
Three-dimensional Lie algebras $\h_1$, $\mathfrak{r}_{3,0}$ and  $\mathfrak{r}'_{3,0}$ admit the following LSA compatible structures, corresponding to the totally real complex structures $J$ on the cotangent Lie algebras $T^* \h_1$, $T^*  \mathfrak{r}_{3,0}$ and 
$T^*  \mathfrak{r}'_{3,0}$ described above:  
$$\begin{array}{lllll} 
 \h_1: \; & e_1 \cdot e_2 = e_3,  \; & e_2 \cdot e_2 =e_1; & & \\[4pt]
\mathfrak{r}_{3,0} : \; & e_1 \cdot e_1 = -e_1,  \; & e_2 \cdot e_1 =-e_2; &  &\\[4pt]
\mathfrak{r}'_{3,0}: \; & e_1 \cdot e_2 = -e_3,  \; & e_1 \cdot e_3 =e_2, \;  & e_2 \cdot e_2 =-\varepsilon e_1,\;  & e_3 \cdot e_3 =-\varepsilon e_1.
\end{array}$$
\end{prop}





\end{document}